\documentclass[11pt,reqno]{amsart}
\usepackage{eurosym}
\usepackage{amsmath,amstext,amssymb,amscd,hyperref}
\usepackage{verbatim}
\usepackage{enumerate}
\usepackage{mathrsfs}
\usepackage[dvipsnames,usenames]{xcolor}
\usepackage{amsthm}

\newtheorem{theorem}{Theorem}[section]
\newtheorem{lemma}[theorem]{Lemma}
\newtheorem{corollary}[theorem]{Corollary}
\newtheorem{proposition}[theorem]{Proposition}

\theoremstyle{definition}
\newtheorem{definition}[theorem]{Definition}

\theoremstyle{remark}
\newtheorem{remark}[theorem]{Remark}
\numberwithin{equation}{section}

\begin{document}
\author[Y. Guo]{Yanqiu Guo}
\address{Department of Mathematics \& Statistics \\
Florida International University\\
Miami, Florida 33199, USA}
\title[Inertial manifolds for the 2D hyperviscous NSE]{Inertial manifolds for the two-dimensional hyperviscous Navier-Stokes equations}
\date{January 25, 2024}
\maketitle

\begin{abstract}
This study establishes the existence of inertial manifolds for the hyperviscous Navier-Stokes equations (HNSE) on a 2D periodic domain:
\begin{equation*}
\partial_t u+ \nu(-\Delta) ^{\beta}u+(u\cdot \nabla )u+\nabla p=f, \;\; \text{on} \;\; \mathbb{T}^2,
\end{equation*}
with $\nabla \cdot u=0$, for any $\beta > \frac{17}{12} $. 
The exponent $\beta = \frac{3}{2}$ is identified as the ``critical" value for the inertial manifold problem in 2D HNSE, below which the spectral gap condition is not satisfied. A breakthrough in this work is that it extends the theory to ``supercritical" regimes where $\beta < \frac{3}{2}$. An important aspect of our argument involves a refined analysis on the sparse distribution of lattice points in annular regions.

\end{abstract}

\section{Introduction}

\subsection{The problem}

The existence of an inertial manifold (IM) for the Navier-Stokes equations (NSE) is a long-standing open problem. It asks whether the solutions of the NSE approach a finite-dimensional, invariant, and Lipschitz-continuous manifold in the phase space at an exponential rate as $t \rightarrow \infty$. This question is significant in both theory and practice because if the NSE has an IM, then, at large time, solving the NSE can be reduced to solving a system of a finite number of ODEs. However, the existence of IMs is still unknown for NSE in either two or three dimensions. Global well-posedness of solutions and the existence of global attractors are available for 2D NSE. Although the 2D theory for NSE is complete in many aspects, the current tools to establish IM are not applicable to 2D NSE. In 3D, even the global well-posedness is not known for NSE.

In this paper, we consider the IM problem for a regularized NSE on a 2D periodic domain:
\begin{equation} \label{HNSEb}
\partial_t u+ \nu(-\Delta) ^{\beta}u+(u\cdot \nabla )u+\nabla p=f, \;\; \text{on} \;\; \mathbb{T}^2 = [0,2\pi]^2,
\end{equation}
with the divergence-free condition $\nabla \cdot u=0$, where $\beta\geq 1$. If $\beta=1$, it is the original NSE. If $\beta>1$, $\nu(-\Delta) ^{\beta}u$ is a strengthened viscous term, 
and the system (\ref{HNSEb}) can be called the hyperviscous Navier-Stokes equations (HNSE). 
The hyperviscous term $\nu(-\Delta) ^{\beta}u$ provides stronger regularization than the original viscous term. 
If $u = \sum_{j\in \mathbb Z^2} \hat u_j e^{ij \cdot x}$, 
then $(-\Delta) ^{\beta}u = \sum_{j \in \mathbb Z^2} |j|^{2\beta} \hat u_j e^{i j \cdot x}$. The global well-posedness of the 2D HNSE (\ref{HNSEb}) is valid when $\beta\geq 1$.

The existence of IMs for the 2D HNSE (\ref{HNSEb}) was shown in \cite{Temam, FST} if $\beta>2$, and in \cite{GG} if $\beta \geq \frac{3}{2}$. 
We refer to $\beta=\frac{3}{2}$ as the ``critical" exponent for the IM problem of the 2D HNSE because the spectral gap condition is satisfied when $\beta\geq \frac{3}{2}$. 
Indeed, in the Hilbert space $L^2(\mathbb T^2)$, the Laplacian $-\Delta$ has eigenfunctions $e^{(j_1 x_1 + j_2 x_2) i}$, $j_1,j_2 \in \mathbb Z$, 
with eigenvalues $j_1^2 + j_2^2$. That is, the set of eigenvalues is the set of sums of two squares $\{j_1^2 + j_2^2\}$.
If we sort the eigenvalues of $-\Delta$ as $0\leq \lambda_1 \leq \lambda_2 \leq \cdots$,  repeated
according to their multiplicities, then
the spectral gap condition for the 2D HNSE (\ref{HNSEb}) demands that there exists $N\in \mathbb N$ such that the quantity
$\frac{\lambda_{N+1}^{\beta} - \lambda_N^{\beta}}{\lambda_{N+1}^{1/2} + \lambda_N^{1/2}}$ is sufficiently large. This is obvious true if $\beta>\frac{3}{2}$. If $\beta=\frac{3}{2}$, the spectral gap condition is still available thanks to the fact that there are arbitrarily large gaps between sums of two squares (see Erd\"os \cite{Erdos}, and Richards \cite{Richards}). The main contribution of this paper is to show the existence of IMs for the 2D HNSE (\ref{HNSEb}) in the ``supercritical" regime where $\beta < \frac{3}{2}$. Under such a scenario, the spectral gap condition is not available. Instead, we explore a certain type of sparse distribution of lattice points in annular regions. Also, the spatial averaging principle plays a crucial role in the proof. It was employed by Mallet-Paret and Sell \cite{Mallet-Sell} to show the existence of IMs for the 3D reaction-diffusion equation. While our focus is on the 2D HNSE, it is worth mentioning that the IM for the 3D HNSE has been obtained in \cite{GG} for $\beta \geq \frac{3}{2}$ by the spatial averaging principle.

\subsection{The literature}
The existence of IMs typically depends on the fulfillment of a spectral gap condition within dissipative PDEs. This condition is satisfied in several notable instances, such as the 2D reaction-diffusion equation, the 2D Cahn–Hilliard equation, and the 1D Kuramoto–Sivashinsky equation. Thus, these equations are known to possess IMs \cite{FST, FNST}. In some cases, the spectral gap condition holds when the nonlinearity in the equation does not involve derivatives, as seen in the reaction-diffusion equation. Alternatively, it can be fulfilled when there is substantial dissipation present in the PDEs, such as in the Cahn–Hilliard equation and the Kuramoto–Sivashinsky equation. Also, it is commonly needed that the physical dimension of the domain for these equations does not exceed two.

When the spectral gap condition is not met, it is sometimes possible to apply a specific transformation to eliminate derivatives in the nonlinearity, thereby reformulating the equation to satisfy the spectral gap condition. This approach is exemplified in the case of diffusive Burgers equations in one or two dimensions, where the Cole-Hopf transform is employed to derive the IMs \cite{Vuk2}. Also, a nonlinear and nonlocal transformation can be effectively employed on the Smoluchowski Equation on a sphere \cite{Vuk}. This transformation removes the gradient from the nonlinear term, leading to the existence of IMs. However, it is worth mentioning that the Kwak transformation applied to the NSE does not confirm the existence of an IM for the NSE. This is primarily due to the transformed linear operator lacking the self-adjoint property \cite{Kost-Zelik2}.

A different approach to obtaining IMs in the absence of the spectral gap condition is the spatial averaging method. This technique leverages a number-theoretical result concerning the sparse distribution of lattice points within spherical shells in $\mathbb R^3$. Mallet-Paret and Sell pioneered the application of this method to show the existence of IMs for 3D reaction-diffusion equations \cite{Mallet-Sell}. Subsequently, it has been adapted to acquire IMs for other models, including the 3D Cahn-Hilliard equation \cite{Kost-Zelik} and 3D regularized NSEs \cite{Kostianko, GG, KLSZ}. Using a combination of spatial and time averaging mechanisms, an IM has been obtained for the 3D complex Ginzburg-Landau equations \cite{KSZ}. Furthermore, the study in \cite{WZM} focused on the HNSE with time-varying forcing, obtaining a locally forward invariant, pullback exponentially attracting, and finite-dimensional Lipschitz manifold.

 In this paper, we employ and further advance the spatial averaging method to investigate the two-dimensional HNSE (\ref{HNSEb}) with a ``supercritical" exponent on the Laplacian.

\vspace{0.1 in}

\section{The main result}

Before presenting our results, it is essential to introduce the concept of inertial manifolds. This notion was initially proposed by Foias, Sell, and Temam in \cite{FST}. 
Below we state a definition of IM that is stronger than the original version in \cite{FST}.

Consider $H$ as a Hilbert space endowed with an orthonormal basis $\{e_j\}_{j=1}^{\infty}$. Let $P_N H$ denote the subspace of $H$ spanned by $\{e_j\}_{j=1}^{N}$, and $Q_N H$ denote the subspace spanned by $\{e_j\}_{j={N+1}}^{\infty}$.

\begin{definition}
\label{defman} Let $H$ be a Hilbert space with the norm $\|\cdot\|$. A subset $\mathcal{M}\subset H$
is called an \emph{inertial manifold} for a dynamical system in $H$
associated with the semigroup $S(t)$, provided the following conditions are
satisfied:

\begin{enumerate}
\item $\mathcal{M}$ is invariant, i.e., $S(t)\mathcal{M}= \mathcal{M}$, for all $t\geq 0$.

\item $\mathcal{M}$ is a finite-dimensional Lipschitz manifold, i.e. there
exists a Lipschitz continuous function $\Phi :P_{N}H\rightarrow Q_{N}H$ such
that the manifold $\mathcal M$ is the graph of $\Phi$, that is, 
\begin{equation*}
\mathcal{M} =\{u \in H: \; u = p+\Phi (p),\,p\in P_{N}H\}.
\end{equation*}

\item The exponential tracking property holds, namely, there exist constants 
$C,\alpha>0$ such that for every $u_0\in H$, there is a corresponding $v_0\in \mathcal{M}$ with
\begin{equation} \label{tracking}
\|S(t)u_0-S(t)v_0\|\leq Ce^{-\alpha t}\|u_0-v_0\|,\text{ for all }t\geq 0.
\end{equation}
\end{enumerate}
\end{definition}

\vspace{0.1 in}

For the analysis of the HNSE (\ref{HNSEb}), we define the phase space as 
\begin{align} \label{H-space}
\mathbb H=\{u\in (L^{2}(\mathbb{T}^{2}))^{2}:\int_{\mathbb{T}^{2}}u\,dx=0,\; \nabla \cdot u=0\}.  
\end{align}
The Sobolev space
$H^{s}(\mathbb T^2)=\left\{ u\in \mathbb H:\|u\|_{H^{s}}^{2}=\sum_{j\in \mathbb{Z}^{2}\backslash \{0\}} |j|^{2s}|\hat{u}_{j}|^{2}<\infty \right\}$, for any $s\in \mathbb R$.

We now state the main result of this manuscript.
\begin{theorem} \label{thm-main}
Consider the Hilbert space $\mathbb H$ defined in (\ref{H-space}).
Assume $\beta > \frac{17}{12}$. Let $f\in H^{\frac{1}{6}}(\mathbb T^2)$.
Then the ``prepared" equation (\ref{prepared}) for the hyperviscous Navier-Stokes equations (\ref{HNSEb}) has an inertial manifold in $\mathbb H$ in the
sense of Definition \ref{defman}.
\end{theorem}

\begin{remark} \label{remk-pre}
In the theory of inertial manifolds, it is customary to modify the original PDE outside its absorbing ball, focusing instead on the IMs of these adjusted equations. Such adaptations are called ``prepared" equations. Notably, within the absorbing ball, the solutions of both the original and the ``prepared" equation coincide, resulting in identical long-term behaviors.
The construction of the ``prepared" equation is detailed in Subsection \ref{subsec-modi}. As specified in Theorem \ref{thm-main}, our focus is on the ``prepared" version (\ref{prepared}) corresponding to the HNSE (\ref{HNSEb}), which possesses an IM.
\end{remark}

\begin{remark}
If $\beta$ is above its critical value, i.e., $\beta \geq \frac{3}{2}$, the existence of IMs for the HNSE (\ref{HNSEb}) has been proved in \cite{GG}
by Gal and Guo. But, the present paper is devoted to handling the supercritical case $\frac{17}{12} < \beta < \frac{3}{2}$.
\end{remark}

\vspace{0.1 in}

\section{The proof of the main result}
\subsection{An abstract model}
We consider the following abstract model in a Hilbert space $H$:
\begin{equation}   \label{a1}
\partial_ t u+A^{\beta}u+A^{1/2}F(u)=f,    \;\;   \text{for} \;\;     \frac{17}{12} < \beta < \frac{3}{2},
\end{equation}
where $f\in H$.
We assume $\beta$ in a ``supercritical" range $\frac{17}{12} < \beta < \frac{3}{2}$ for the above abstract model where the spectral gap condition is not available. Here, the nonlinear function $F: H \rightarrow H$ is globally Lipschitz continuous with Lipschitz constant $L$.
The abstract model (\ref{a1}) is motivated by the HNSE (\ref{HNSEb}).

Here, $A:\mathcal{D}(A)\rightarrow H$ is a linear, symmetric, and positive operator
with a compact inverse. The operator $A$ possesses a complete orthonormal
system of eigenvectors $\{e_{j}\}_{j=1}^{\infty }$ in $H$, corresponding to eigenvalues $\lambda _{j}$, which satisfy $\lambda _{j}\rightarrow \infty $ as 
$j\rightarrow \infty $ and $Ae_{j}=\lambda _{j}e_{j}$ with $0<\lambda _{1}\leq \lambda _{2}\leq \lambda
_{3}\leq \cdots$.
 For a given $N\in \mathbb{N}$, we define projection operators $P_N$ and $Q_N$ on the lower and higher modes, respectively, as
\begin{equation}
P_{N} u =\sum_{j=1}^{N}u_{j}e_{j},\;\;\;Q_{N}u=\sum_{j=N+1}^{\infty}u_{j}e_{j},  \label{proj}
\end{equation}
where $u_{j}=(u,e_{j})$.

In $H$, the norm is denoted by $\|\cdot \|$ and the inner product by $\left( \cdot ,\cdot \right)$. Additionally, the projection operators on the low, high, and intermediate Fourier modes are defined as follows:
\begin{equation*}
\mathcal P_{k,N}u =\sum_{\lambda _{j}<\lambda
_{N}-k}u_{j}e_{j},\;\;\;\mathcal Q_{k,N}u =\sum_{\lambda _{j}>\lambda
_{N}+k}u_{j}e_{j},\;\;\;\mathcal I_{k,N}u =\sum_{\lambda _{N}-k\leq \lambda _{j}\leq
\lambda _{N}+k}u_{j}e_{j},
\end{equation*}%
where $u_{j}=(u,e_{j})$, for some $k<\lambda _{N}$.

The following elementary inequality is useful:
\begin{align} \label{c0}
a^{\beta} - b^{\beta} \geq \frac{1}{2}(a-b) (a^{\beta-1} + b^{\beta-1}),
\end{align}
for all real numbers $a\geq b\geq 0$, and for $\beta\geq 1$.

We consider solutions $u$ of equation (\ref{a1}) belonging to the space $$u\in L^2(0,T; \mathcal D(A^{\beta/2})) \cap C([0,T];H) \;\;\text{with}\;\; \partial_t u \in L^2(0,T; \mathcal D(A^{\beta/2})').$$

\begin{lemma}

\label{thm1} 
Assume $\frac{17}{12} < \beta < \frac{3}{2}$ and  $s \in (3-2\beta, \,\frac{1}{6})$.
Let $u_{1}, u_{2} $ be two solutions of equation (\ref{a1}) for all $t\geq 0$. 
Define $v = u_{1} - u_{2}$ and let $V(t) = \|q\|^2 - \|p\|^2$, where $p = P_{N}v$ and $q = Q_{N}v$.
Assume $F:H\rightarrow H$ is Gateaux differentiable and satisfies $\|F^{\prime }(u)\|_{\mathcal{L}(H,H)}\leq L$ for all $u\in H$. Further, assume the existence of arbitrarily large $\lambda_N$ and $k \geq \lambda_N^s$, with $1\leq \lambda _{N+1}-\lambda_{N}  \leq \frac{k}{2}$, such that a spatial averaging condition is fulfilled: 
\begin{equation}
\|\mathcal I_{k,N}F^{\prime }(u)\mathcal I_{k,N} \|_{\mathcal L(H,H)} \leq    \frac{1}{16} \lambda_N^{ -   \frac{1}{2}(3-2\beta)},\text{\ \ for all\ \ }u\in H.
\label{averaging}
\end{equation}
Under these conditions, the following strong cone property holds:
\begin{equation}
\frac{d}{dt} V(t)+\left( \lambda _{N+1}^{\beta}+\lambda _{N}^{\beta}\right)
V(t)\leq -\frac{\lambda^{\beta-1} _{N}}{8}\|v(t)\|^{2}, \;\;\;\text{for all}  \;\;t\geq 0.
  \label{stcone}
\end{equation}
\end{lemma}

\begin{proof}
Let us consider $u_{1}, u_{2}$ as two solutions of (\ref{a1}) and define $v=u_{1}-u_{2}$. Subsequently, we obtain the following equation:
\begin{equation}
\partial_t v+A^{\beta}v+A^{1/2}[F(u_{1})-F(u_{2})]=0,  \label{a2'}
\end{equation}
with $\frac{17}{12} < \beta < \frac{3}{2}$.

Define $p=P_{N}v$ and $q=Q_{N}v$. Taking the duality pairing of (\ref{a2'}) with $p $ and $q$ respectively, we derive
\begin{equation}
\begin{cases}
\frac{1}{2}\frac{d}{dt}\|p\|^{2} + \|A^{\beta/2}p\|^{2}+(F(u_{1})-F(u_{2}),A^{1/2}p)=0,
\\ 
\frac{1}{2}\frac{d}{dt} \|q\|^{2} + \|A^{\beta/2}q\|^{2}+(F(u_{1})-F(u_{2}),A^{1/2}q)=0.
\end{cases}
\label{a3}
\end{equation}

Let us denote $V(t)=\|q\|^2-\|p\|^2$. By subtracting the two equations in (\ref{a3}), we arrive at
\begin{equation}
\frac{d}{dt} V(t)=-2(\|A^{\beta/2}q\|^{2}-\|A^{\beta/2}p\|^{2})+2(F(u_{1})-F(u_{2}),A^{1/2}p-A^{1/2}q).
\label{a7}
\end{equation}

Thanks to the fundamental theorem of calculus for the Gateaux derivative, it holds that $F(u_{1})-F(u_{2})=\int_{0}^{1}F^{\prime }(su_{1}+(1-s)u_{2})vds$ where $v=u_1 - u_2$. Therefore, by setting $\alpha=\frac{\lambda _{N+1}^{\beta}+\lambda _{N}^{\beta}
}{2}$, we obtain from (\ref{a7}) that
\begin{align}   \label{a9}
& \frac{d}{dt} V(t)+2\alpha V(t)=\left[ \alpha
V(t)-(\|A^{\beta/2}q\|^{2}-\|A^{\beta/2}p\|^{2})\right] -(\|A^{\beta/2}q\|^{2}-\alpha \|q\|^{2})  \notag\\
& -(\alpha \|p\|^{2}-\|A^{\beta/2}p\|^{2})+2\int_{0}^{1}\left( F^{\prime
}(su_{1}+(1-s)u_{2})v,A^{1/2}p-A^{1/2}q\right) ds. 
\end{align}

Next, we will estimate each term on the right-hand side of (\ref{a9}).

Recognizing that $\|v\|^{2}=\|p\|^{2}+\|q\|^{2}$, along with $\|A^{\beta/2}q \| ^{2} \geq \lambda _{N+1}^{\beta} \|q\|^{2}$ and $\|A^{\beta/2}p\|^{2}\leq \lambda _{N}^{\beta}\|p\|^{2}$, we infer the following:
\begin{align}  \label{a8'}
\alpha V(t)-(\|A^{\beta/2}q\|^{2}-\|A^{\beta/2}p\|^{2})  \leq -\frac{\lambda _{N+1}^{\beta}-\lambda _{N}^{\beta}}{2}\|v\|^{2}.  
\end{align}

By employing inequality (\ref{c0}) and the assumption $\lambda _{N+1}-\lambda _{N}\geq 1$, we get
\begin{equation} \label{a101}
\frac{\lambda _{N+1}^{\beta}-\lambda _{N}^{\beta}}{2}
\geq    \frac{1}{4}(\lambda_{N+1}-\lambda _{N})(\lambda _{N+1}^{\beta-1}+\lambda _{N}^{\beta-1})
\geq \frac{1}{2} \lambda _{N}^{\beta-1}.  
\end{equation}

Applying (\ref{a101}) to (\ref{a8'}) leads to
\begin{align}  \label{a8}
\alpha V(t)-(\|A^{\beta/2}q\|^{2}-\|A^{\beta/2}p\|^{2}) \leq  -  \frac{1}{2} \lambda_N^{\beta-1} \|v\|^2.
\end{align}

By splitting $p=\mathcal P_{k,N}v+\mathcal I_{k,N}p$, and noting that $\|\mathcal P_{k,N}A^{\beta/2}v\|^2\leq \left( \lambda _{N}-k\right)
^{\beta} \| \mathcal P_{k,N}v\|^{2}$ and $\|\mathcal I_{k,N}A^{\beta/2}p\|^{2}\leq \lambda _{N}^{\beta}\|\mathcal I_{k,N}p\|^2$, we deduce
\begin{align}  \label{a10}
& \alpha \|p\|^2-\|A^{\beta/2}p\|^2  \geq \frac{\lambda _{N+1}^{\beta}  + \lambda _{N}^{\beta}  }{2} \left(
\|\mathcal P_{k,N}v\|^2+\|\mathcal I_{k,N}p\|^2\right) -\left(
\|\mathcal P_{k,N}A^{\beta/2}v\|^2+\|\mathcal I_{k,N}A^{\beta/2}p\|^2\right)  \notag \\
& \geq \Big[ \lambda _{N}^{\beta}-(\lambda_{N}-k)^{\beta}\Big] \|\mathcal P_{k,N}v\|^2+\frac{\lambda _{N+1}^{\beta}-\lambda_{N}^{\beta}}{2}\|\mathcal I_{k,N}p\|^2  \notag\\
&\geq  \frac{1}{2} k\lambda_{N}^{\beta-1}\|\mathcal P_{k,N}v\|^2 +  \frac{1}{2} \lambda _{N}^{\beta-1}\|\mathcal I_{k,N}p\|^2,
\end{align}
where we have used inequality (\ref{c0}).

Given $q=\mathcal Q_{k,N}v+\mathcal I_{k,N}q$, we write 
\begin{align*} 
\|A^{\beta/2}q\|^2-\alpha \|q\|^2  =
(\|\mathcal Q_{k,N}A^{\beta/2}v\|^2-\alpha\|\mathcal Q_{k,N}v\|^{2})
+(\|\mathcal I_{k,N}A^{\beta/2}q\|^2-\alpha \|\mathcal I_{k,N}q\|^2).
\end{align*}
Since $\|\mathcal I_{k,N}A^{\beta/2}q\|^2\geq \lambda_{N+1}^{\beta}\|\mathcal I_{k,N}q\|^2$, and applying (\ref{a101}), we get
\begin{align*}  
\|\mathcal I_{k,N}A^{\beta/2}q\|^2-\alpha \|\mathcal I_{k,N}q\|^2 \geq \frac{\lambda_{N+1}^{\beta} -
\lambda_N^{\beta} }{2} \|\mathcal I_{k,N}q\|^2 \geq \frac{1}{2}\lambda_N^{\beta-1} \|\mathcal I_{k,N}q\|^2.
\end{align*}
Combining the above two formulas, we have
\begin{align} \label{b4}
\|A^{\beta/2}q\|^2-\alpha \|q\|^2  \geq  (\|\mathcal Q_{k,N}A^{\beta/2}v\|^2-\alpha\|\mathcal Q_{k,N}v\|^{2}) +  \frac{1}{2}\lambda_N^{\beta-1} \|\mathcal I_{k,N}q\|^2.
\end{align}
Note that $\|\mathcal Q_{k,N}A^{\beta/2}v\|^2\geq (\lambda_{N}+k)^{\beta}\|\mathcal Q_{k,N}v\|^{2}$ and $\alpha \leq  \lambda _{N+1}^{\beta}$. Thus
\begin{align}
&\|\mathcal Q_{k,N}A^{\beta/2}v\|^2-\alpha \|\mathcal Q_{k,N}v\|^{2}\geq \left[ (\lambda_{N}+k)^{\beta}-\lambda _{N+1}^{\beta}\right] \|\mathcal Q_{k,N}v\|^{2}  \label{b6} \notag\\
& \geq \frac{1}{2}(\lambda _{N}+k-\lambda _{N+1})\left[ (\lambda
_{N}+k)^{\beta-1}+\lambda _{N+1}^{\beta-1}\right] \|\mathcal Q_{k,N}v\|^{2} \geq \frac{1}{2} k\lambda _{N+1}^{\beta-1}\|\mathcal Q_{k,N}v\|^{2},
\end{align}
because of the assumption that $\lambda _{N+1}-\lambda_{N}\leq  \frac{k}{2}$.

In addition to (\ref{b6}), it is necessary to establish another lower bound for $\|\mathcal Q_{k,N}A^{\beta/2}v\|^2-\alpha
\|\mathcal Q_{k,N}v\|^{2}$. 
Given the assumption that  $s>3-2\beta$, we can take a number $\gamma \in (1-s, \, 2\beta -2)$. Accordingly, we decompose the expression as follows:
\begin{align}
& \|\mathcal Q_{k,N}A^{\beta/2}v\|^2-\alpha \|\mathcal Q_{k,N}v\|^{2}  \label{b1} \notag \\
& \geq \frac{2}{\lambda _{N}^{\gamma}}\|\mathcal Q_{k,N}A^{\beta/2}v\|^2+
\Big[ \Big( 1-\frac{2}{\lambda _{N}^{\gamma}}\Big)
\|\mathcal Q_{k,N}A^{\beta/2}v\|^2-\alpha \|\mathcal Q_{k,N}v\|^{2}\Big].
\end{align}

Since $\alpha\leq \lambda _{N+1}^{\beta}  \leq (\lambda _{N} + \frac{1}{2}k)^{\beta}$, we derive
\begin{align}
& \Big( 1-\frac{2}{\lambda _{N}^{\gamma}} \Big)
\|\mathcal Q_{k,N}A^{\beta/2}v\|^2-\alpha \|\mathcal Q_{k,N}v\|^{2}  \label{b2} \notag \\
& \geq \Big[ \Big( 1-\frac{2}{\lambda _{N}^{\gamma}}\Big)
(\lambda _{N}+k)^{\beta}-(\lambda _{N} + \frac{1}{2}k)^{\beta}\Big]
\|\mathcal Q_{k,N}v\|^{2}.
\end{align}

Applying the binomial theorem and considering the assumption that $k\geq \lambda_N^s$, we find
 \begin{align}     \label{b2'}
\Big( 1-\frac{2}{\lambda _{N}^{\gamma}}\Big)
(\lambda _{N}+k)^{\beta} - (\lambda _{N} + \frac{1}{2}k)^{\beta} 
&\geq \frac{\beta}{2} \lambda_N^{\beta - 1} k  - 2 \lambda_N^{\beta - \gamma} + O(\lambda_N^{\beta - 2}k^2) \notag\\
&\geq  \frac{\beta}{2} \lambda_N^{\beta - 1+s}   - 2 \lambda_N^{\beta - \gamma} + O(\lambda_N^{\beta - 2 + 2s})
\geq 0,
\end{align} 
for sufficiently large $\lambda_N$. This follows from the condition $\gamma > 1-s$, which ensures that the leading term $\frac{\beta}{2} \lambda_N^{\beta - 1+s}$ increases at a faster rate than the secondary term $2 \lambda_N^{\beta - \gamma}$, when $\lambda_N \rightarrow \infty$. The term $O(\lambda_N^{\beta - 2 + 2s})$ represents lower-order contributions.

Combining (\ref{b1}), (\ref{b2}) and (\ref{b2'}), we conclude
\begin{align} \label{b1'}
\|\mathcal Q_{k,N}A^{\beta/2}v\|^2 -  \alpha \|\mathcal Q_{k,N}v\|^{2}    \geq \frac{2}{\lambda _{N}^{\gamma}}\|\mathcal Q_{k,N}A^{\beta/2}v\|^2
 \geq \frac{2}{\lambda _{N}^{\gamma + 1 - \beta}}\|\mathcal Q_{k,N}A^{1/2}v\|^2.
 \end{align}
Applying the estimates (\ref{b6}) and (\ref{b1'}) to inequality (\ref{b4}), we obtain
\begin{equation}
\|A^{\beta/2}q\|^2-\alpha \|q\|^2\geq \frac{1}{\lambda_{N}^{\gamma + 1 - \beta}}\|\mathcal Q_{k,N}A^{1/2}v\|^2+ \frac{1}{4} k \lambda_{N+1}^{\beta-1}\|\mathcal Q_{k,N}v\|^{2} + \frac{1}{2} \lambda _{N}^{\beta-1}\|\mathcal I_{k,N}q\|^2.
\label{b5}
\end{equation}

Merging (\ref{a8}), (\ref{a10}) and (\ref{b5}) yields
\begin{align}  \label{lincom}
&\left[ \alpha
V(t)-(\|A^{\beta/2}q\|^2-\|A^{\beta/2}p\|^2)\right] -(\|A^{\beta/2}q\|^2-\alpha \|q\|^2) - (\alpha \|p\|^2-\|A^{\beta/2}p\|^2) \notag\\
&\leq   - \frac{1}{2} \lambda_N^{\beta-1} \|v\|^2 - \frac{1}{2} k\lambda_{N}^{\beta-1}\|\mathcal P_{k,N}v\|^2- \frac{1}{2}\lambda _{N}^{\beta-1} (\|\mathcal I_{k,N}p\|^2  +   \|\mathcal I_{k,N}q\|^2)  \notag\\
&-\frac{1}{\lambda_{N}^{\gamma + 1 -\beta}}\|\mathcal Q_{k,N}A^{1/2}v\|^2 - \frac{1}{4} k \lambda_{N+1}^{\beta-1}\|\mathcal Q_{k,N}v\|^{2}.
\end{align}

To address the nonlinear term in (\ref{a9}), we analyze $(F^{\prime }(u)v,A^{1/2}p-A^{1/2}q)$ for any $u\in H$ as follows:
\begin{align}
& (F^{\prime }(u)v,A^{1/2}p-A^{1/2}q)  \label{nonlinear}  \notag\\
& =(\mathcal I_{k,N}F^{\prime }(u)\mathcal I_{k,N}v,A^{1/2}p-A^{1/2}q)+(\mathcal I_{k,N}F^{\prime
}(u)\mathcal P_{k,N}v,A^{1/2}p-A^{1/2}q)  \notag \\
& \;\;\;+(\mathcal I_{k,N}F^{\prime }(u)\mathcal Q_{k,N}v,A^{1/2}p-A^{1/2}q)+(\mathcal P_{k,N}F^{\prime
}(u)v,A^{1/2}p-A^{1/2}q)  \notag \\
& \;\;\;+(\mathcal Q_{k,N}F^{\prime }(u)v,A^{1/2}p-A^{1/2}q)  \notag \\
& =(\mathcal I_{k,N}F^{\prime }(u)\mathcal I_{k,N}v,A^{1/2}p-A^{1/2}q)+(F^{\prime
}(u)\mathcal P_{k,N}v,\mathcal I_{k,N}A^{1/2}p)  \notag \\
& \;\;\;-(F^{\prime }(u)\mathcal P_{k,N}v,\mathcal I_{k,N}A^{1/2}q)+(F^{\prime
}(u)\mathcal Q_{k,N}v,\mathcal I_{k,N}A^{1/2}p)-(F^{\prime }(u)\mathcal Q_{k,N}v,\mathcal I_{k,N}A^{1/2}q) 
\notag \\
& \;\;\;+(F^{\prime }(u)v,\mathcal P_{k,N}A^{1/2}v)-(F^{\prime }(u)v,\mathcal Q_{k,N}A^{1/2}v).
\end{align}%
We will estimate every term on the right-hand side of (\ref{nonlinear}).

Given that $\|F^{\prime }(u)\|_{\mathcal{L}(H,H)}\leq L$, we observe
\begin{align}
& |(F^{\prime }(u)\mathcal P_{k,N}v,\mathcal I_{k,N}A^{1/2}p)|+|(F^{\prime
}(u)\mathcal Q_{k,N}v,\mathcal I_{k,N}A^{1/2}p)|  \label{a12}  \notag\\
& \leq L(\|\mathcal P_{k,N}v\|+\|\mathcal Q_{k,N}v\|)\|\mathcal I_{k,N}A^{1/2}p\|  \notag \\
& \leq L\lambda _{N}^{1/2}(\|\mathcal P_{k,N}v\|+\|\mathcal Q_{k,N}v\|)   \|\mathcal I_{k,N}p\|  \notag \\
& \leq    L^{2}\lambda _{N}^{2-\beta}\left( \|\mathcal P_{k,N}v\|^2+\|\mathcal Q_{k,N}v\|^{2}\right) +\frac{1}{4}\lambda _{N}^{\beta-1}\|\mathcal I_{k,N}p\|^2.  
\end{align}%
In a similar manner,
\begin{align}
& |(F^{\prime }(u)\mathcal P_{k,N}v,\mathcal I_{k,N}A^{1/2}q)|+|(F^{\prime
}(u)\mathcal Q_{k,N}v,\mathcal I_{k,N}A^{1/2}q)|  \label{a13} \notag\\
& \leq L(\|\mathcal P_{k,N}v\|+\|\mathcal Q_{k,N}v\|) \|\mathcal I_{k,N}A^{1/2}q\|  \notag \\
& \leq L(\|\mathcal P_{k,N}v\|+\|\mathcal Q_{k,N}v\|)(\lambda _{N}+k)^{1/2}\|\mathcal I_{k,N}q\|  \notag \\
& \leq L(\|\mathcal P_{k,N}v\|+\|\mathcal Q_{k,N}v\|)(2\lambda _{N})^{1/2}  \|\mathcal I_{k,N}q\|  \notag \\
& \leq 4 L^{2}\lambda _{N}^{2-\beta}\left( \|\mathcal P_{k,N}v\|^2+\|\mathcal Q_{k,N}v\|^{2}\right) +    \frac{1}{8}    \lambda _{N}^{\beta-1}\|\mathcal I_{k,N}q\|^2.
\end{align}

Next, 
\begin{align}
& |(F^{\prime }(u)v,\mathcal P_{k,N}A^{1/2}v)|\leq L       \|v\|      \|\mathcal P_{k,N}A^{1/2}v\|
\leq
L \|v\|   \left( \lambda _{N}-k\right) ^{1/2}\|\mathcal P_{k,N}v\|  \label{a14} \notag\\
& \leq L  \|v\| \lambda _{N}^{1/2}\|\mathcal P_{k,N}v\|\leq 
8 L^2\lambda _{N}^{2-\beta}\|\mathcal P_{k,N}v\|^2  + \frac{1}{32} \lambda _{N}^{\beta-1}  \|v\|^2.
\end{align}
Also, 
\begin{align}  \label{a15}
|(F^{\prime }(u)v,\mathcal Q_{k,N}A^{1/2}v)|  \leq L       \|v\|    \|\mathcal Q_{k,N}A^{1/2}v\|   \leq   \frac{1}{32}     \lambda_{N}^{\beta-1}       \|v\|^2  +\frac{8 L^2}{ \lambda _{N}^{\beta-1}}\|\mathcal Q_{k,N}A^{1/2}v\|^2. 
\end{align}
Moreover, applying the spatial averaging condition (\ref{averaging}), we deduce
\begin{align}
& |(\mathcal I_{k,N}F^{\prime }(u)\mathcal I_{k,N}v,A^{1/2}p-A^{1/2}q)|
\leq  \frac{1}{16}   \lambda_N^{ -   \frac{1}{2}(3-2\beta)}
\|v\|   \|A^{1/2}p-A^{1/2}q\|     \notag\\
&\leq \frac{1}{16}    \lambda_N^{ -   \frac{1}{2}(3-2\beta)} \|v\|  \|A^{1/2}p\|+ \frac{1}{16}   \lambda_N^{ -   \frac{1}{2}(3-2\beta)} \|v\|  \|A^{1/2}q\| \label{a16} \notag \\
& \leq \frac{1}{16}   \lambda_N^{ -   \frac{1}{2}(3-2\beta)} \lambda _{N}^{1/2}\|v\|      \|p\| + \frac{1}{16}  \lambda_N^{ -   \frac{1}{2}(3-2\beta)}  \|v\|  \|\mathcal Q_{k,N}A^{1/2}v\| +  \frac{1}{16}  \lambda_N^{ -   \frac{1}{2}(3-2\beta)} \|v\|        \|\mathcal I_{k,N}A^{1/2}q\|  \notag \\
& \leq \frac{1}{16}\lambda _{N}^{\beta-1}\|v\|^2 + \frac{1}{16} \lambda_N^{ -   \frac{1}{2}(3-2\beta)} \|v\|  \|\mathcal Q_{k,N}A^{1/2}v\|+ \frac{1}{16} \lambda_N^{ -   \frac{1}{2}(3-2\beta)}
(2\lambda _{N})^{1/2}\|v\|\|\mathcal I_{k,N}q\|  \notag \\
& \leq \frac{1}{16}\lambda _{N}^{\beta-1}\|v\|^2 + \frac{1}{512} \lambda_{N}^{3\beta - 4}\|v\|^2+  \frac{2}{\lambda_N^{\beta-1}} \|\mathcal Q_{k,N}A^{1/2}v\|^2
+ \frac{1}{64}\lambda _{N}^{\beta-1}\|v\|^2 +   \frac{1}{8}\lambda _{N}^{\beta-1}\|\mathcal I_{k,N}q\|^2   \notag\\
&\leq  \frac{1}{8}\lambda_{N}^{\beta-1}\|v\|^2+  \frac{2}{\lambda_N^{\beta-1}} \|\mathcal Q_{k,N}A^{1/2}v\|^2  + \frac{1}{8}\lambda _{N}^{\beta-1}\|\mathcal I_{k,N}q\|^2,
\end{align}
for sufficiently large $\lambda_N$, where we have used the assumption $\beta<3/2$, leading to $\beta-1 > 3\beta -4$.

Substituting (\ref{a12})-(\ref{a16}) into (\ref{nonlinear}) yields 
\begin{align} \label{non}
&|(F^{\prime }(u)v, A^{1/2}p-A^{1/2}q)| \leq   
13  L^{2}\lambda _{N}^{2-\beta} \|\mathcal P_{k,N}v\|^2+    5 L^{2}\lambda _{N}^{2-\beta}   \|\mathcal Q_{k,N}v\|^{2}
  \notag\\
& +\frac{1}{4}\lambda _{N}^{\beta-1} (\|\mathcal I_{k,N}p\|^2  +  \|\mathcal I_{k,N}q\|^2)    +  \frac{3}{16}  \lambda_{N}^{\beta-1}\|v\|^2+  \frac{8L^2+2}{\lambda_N^{\beta-1}} \|\mathcal Q_{k,N}A^{1/2}v\|^2  ,
\end{align}%
for any $u\in H$.

Owing to (\ref{a9}), (\ref{lincom}) and (\ref{non}), we derive
\begin{align}  \label{a17}
&\frac{d}{dt} V(t)+2\alpha V(t) 
\leq   -\frac{ \lambda _{N}^{\beta-1}}{8}\|v\|^2  - ( \frac{1}{2} k\lambda_{N}^{\beta-1} -   26 L^{2}\lambda _{N}^{2-\beta}) \|\mathcal P_{k,N}v\|^2   \notag\\
& - (     \frac{1}{4} k \lambda_{N+1}^{\beta-1}   -     10  L^{2}\lambda _{N}^{2-\beta})      \|\mathcal Q_{k,N}v\|^{2}
-   \Big( \frac{1}{\lambda_{N}^{\gamma + 1 -\beta}}  -   \frac{16L^2+4}{\lambda_N^{\beta-1}}  \Big)  \|\mathcal Q_{k,N}A^{1/2}v\|^2.
\end{align}
Since $k\geq \lambda_N^s$ with $s>3-2\beta$, 
we have $k\lambda_{N}^{\beta-1} \geq \lambda_N^{s+\beta -1}$ where $s+\beta-1>2-\beta$. Therefore, $\frac{1}{2} k \lambda_{N}^{\beta-1} \geq  \frac{1}{2}   \lambda_N^{s+\beta -1} \geq   26  L^{2}\lambda _{N}^{2-\beta}$
for sufficiently large $\lambda_N$. Additionally, given $\gamma < 2\beta -2$, it follows that $\gamma + 1 -\beta < \beta -1$, implying that 
$\frac{1}{\lambda_{N}^{\gamma + 1 -\beta}} >  \frac{16L^2+4}{\lambda_N^{\beta-1}}$ for sufficiently large $\lambda_N$. 
Consequently, from (\ref{a17}), we can conclude that
\begin{equation}  \label{a18}
\frac{d}{dt} V(t)+2\alpha V(t)\leq -\frac{ \lambda _{N}^{\beta-1}}{8}\|v\|^2,  \;\; \text{for all} \;\; t\geq 0,
\end{equation}
for sufficiently large $\lambda _{N}$.
\end{proof}

\vspace{0.1 in}

\begin{remark}
The strong cone property, as expressed in (\ref{a18}), has been recognized as a fundamental condition for the existence of IMs. 
Generally, the \emph{strong cone property} is formulated as follows:
\begin{align} \label{strong-cone}
\frac{d}{dt} V(t)+ \alpha V(t)\leq - \mu \|v\|^2,  \;\; \text{for all} \;\;  t\geq 0,
\end{align}
where $\alpha$ and $\mu$ are positive constants. This concept was originally introduced in \cite{Zelik14}.
Recall that $V(t)=\|q\|^2-\|p\|^2$, where $p=P_{N}v$ and $q=Q_{N}v$. The function $v=u_1-u_2$ is the difference between two solutions $u_1,u_2$ of the abstract equation (\ref{a1}).
For further details and proofs regarding how the strong cone property, as stated in (\ref{strong-cone}), combined with the boundedness condition $\|F(u)\|\leq C$ for all $u\in H$, imply the existence of IMs, readers are referred to \cite{GG, Kost-Zelik, Zelik14}.  Also, the value of $\alpha$ in the strong cone property (\ref{strong-cone}) coincides with the value of $\alpha$ in the exponential tracking property (\ref{tracking}). Consequently, we have the following corollary regarding the existence of IMs for the abstract model (\ref{a1}). 
\end{remark}

\begin{corollary}
\label{cor1}
Consider that the conditions outlined in Lemma \ref{thm1} are fulfilled, and assume $\|F(u)\|\leq C$ for all $u\in H$.
Under these premises, it follows that problem (\ref{a1}) admits an $N$-dimensional inertial manifold $\mathcal{M}$, as defined in Definition \ref{defman}.
\end{corollary}

\vspace{0.1 in}

\subsection{Sparse distribution of lattice points in annuli}

To validate the spatial averaging condition (\ref{averaging}) for the HNSE, we investigate a particular property concerning the sparse distribution of lattice points within annular regions in $\mathbb{R}^2$. This exploration is motivated by and builds upon number-theoretical insights presented in \cite{Mallet-Sell}.

\begin{lemma}  \label{thm2D}
Assume $0<s<\frac{1}{6}$. There exist arbitrarily large $\lambda$ and $k \geq C \lambda^s$ for some constant $C$ independent of $\lambda$ and $k$, such that,  
any two lattice points $n,\ell \in \mathbb Z^2$ that belong to the annulus $\{x\in \mathbb R^2: \lambda - k \leq |x|^2 \leq \lambda + k\}$ must satisfy $|n-\ell|> \lambda^{s/2}$.
\end{lemma}

\vspace{0.05 in}

\begin{remark}
The annuli described in Lemma \ref{thm2D} may contain no lattice points, a single lattice point, or multiple lattice points. In cases where multiple lattice points exist within such an annulus, each pair of points is separated by a sufficiently large distance.
\end{remark}

\vspace{0.05 in}

\begin{proof}
We draw ideas from \cite{Mallet-Sell, GI} and begin by introducing a notation. For functions $f(x)$ and $g(x)$, we denote
$f(x) \sim g(x)$ to mean $\lim_{x\rightarrow \infty} \frac{f(x)}{g(x)} = 1$. 

Consider a family of disjoint annuli in $\mathbb R^2$ defined as: 
\begin{align*}
N_m^{\mu} = \{x\in \mathbb R^2:  \mu+ m \kappa < |x|^2 \leq \mu + (m+1) \kappa\},
\end{align*}
where $m \in \mathbb Z$ in the range $0\leq m \leq  J = \lfloor \mu^{1/2} \rfloor $,
and we set
$$\kappa = \mu^s, \;\;\text{where}\;\;  0<s<\frac{1}{6}.$$

Our goal is to show that, for sufficiently large $\mu$, there exists $m\in [0,J]$ such that $N_m^{\mu}$ does not contain any pair of lattice points at a distance less than or equal to $\mu^{s/2}$.

The union of these annuli $N_m^{\mu}$ is denoted as
\begin{align*}
N^{\mu} = \bigcup_{m=0}^{J} N_m^{\mu} =
\{x \in \mathbb R^2: \;  \mu < |x|^2 \leq \mu + (J+1)\kappa\}.
\end{align*}
Evidently, $N^{\mu}$ forms an annulus in $\mathbb R^2$.

The thickness of $N^{\mu}$ is given by $\sqrt{\mu + (J+1)\kappa} - \sqrt{\mu}$. 
With $J=  \lfloor \mu^{1/2} \rfloor $ and $\kappa = \mu^s$ for $0<s<\frac{1}{6}$, a straightforward calculation shows that, as $\mu\rightarrow \infty$,
\begin{align} \label{thickN}
\text{the thickness of} \,N^{\mu} \sim \frac{1}{2} \mu^s, \;\text{namely},\; \lim_{\mu\rightarrow \infty} \frac{\text{thickness of} \,N^{\mu} }{\frac{1}{2}\mu^s} = 1.
\end{align}

Let $\ell, n  \in \mathbb Z^2$ be two distinct lattice points within the same annulus $N^{\mu}_m$:
\begin{align*}
\ell, n  \in N^{\mu}_m  \;\;\text{such that} \;\; 0<| \ell - n| \leq  \mu^{s/2},
\end{align*}
for some $m \in [0,J]$. Denoting $j= \ell-n$, we have $0<|j| \leq  \mu^{s/2} $. Since $|\ell|^2 = |n|^2 + 2 n \cdot j + |j|^2$, it follows that
\begin{align*}
|n \cdot j|\leq \frac{1}{2} \left| |n|^2 - |\ell|^2 \right| + \frac{1}{2} |j|^2 < \frac{1}{2} \kappa + \frac{1}{2} \mu^s = \mu^s,
\end{align*}
as $\ell,  n  \in N^{\mu}_m$. Since $\ell$ and $n$ are interchangeable, we have also
$|\ell \cdot j| < \mu^s$. Therefore, the lattice points $n$ and $\ell$ belong to a strip 
\begin{align} \label{strip}
S_j^{\mu} = \{x \in \mathbb R^2: |x \cdot j|< \mu^s\}
\end{align}
for some $j\in \mathbb Z^2$ satisfying $0<|j| \leq \mu^{s/2}$. The strip $S_j$ is symmetric about the line $x \cdot j =0$ in $\mathbb R^2$.

We denote $S^{\mu}$ as the finite union $S^{\mu} = \bigcup_{0<|j| \leq \mu^{s/2}} S^{\mu}_j$.
Note that the set $S^{\mu}$ contains all pairs of lattice points $\ell, n \in \mathbb Z^2$ at a distance less than or equal to $\mu^{s/2}$ belonging to an annulus $N^{\mu}_m$ for some integer $m \in [0,J]$.
In other words,
\begin{align} \label{Smusup}
S^{\mu} \supset \left\{  \ell, n \in \mathbb Z^2:  \,  0< |\ell -n  | \leq  \mu^{s/2} \;\text{with} \; \ell, n \in N^{\mu}_m \; \text{for some integer} \; m \in [0,J] \right\}.
\end{align}

From (\ref{strip}), we note that
\begin{align} \label{widS}
\text{the width of} \; S^{\mu}_j \leq 2\mu^s. 
\end{align}

Also, as $\mu \rightarrow \infty$, asymptotically,
\begin{align} \label{measSN}
\text{meas}(S_j^{\mu} \cap N^{\mu}) \sim 2 (\text{width of $S^{\mu}_j$})(\text{thickness of $N^{\mu}$}).
\end{align}
Here, ``meas" stands for the Lebesgue measure of a set in $\mathbb R^2$. 

Combining (\ref{thickN}), (\ref{widS}) and (\ref{measSN}), it follows that, for sufficiently large $\mu$, 
\begin{align} \label{measSN2}
\text{meas}(S_j^{\mu} \cap N^{\mu})  \leq c \mu^{2s},
\end{align}
for some constant $c$.

Given that $S^{\mu} = \cup_{0<|j| \leq \mu^{s/2}} S^{\mu}_j$ is a finite union, and the count of $j\in \mathbb Z^2$ satisfying $|j| \leq \mu^{s/2}$ is asymptotically of the order $\mu^s$, 
it follows that, for sufficiently large $\mu$, according to inequality (\ref{measSN2}), we have
 \begin{align}  \label{measSN3}
\text{meas}(S^{\mu} \cap N^{\mu}) \leq  C\mu^{3s}.
\end{align}

Note, as $\mu \rightarrow \infty$, we have
\begin{align} \label{cardSN}
\text{card}(S^{\mu} \cap N^{\mu} \cap \mathbb Z^2) \sim \text{meas} (S^{\mu}\cap N^{\mu}).
\end{align}
Here, ``card" represents the cardinality of lattice points.

By (\ref{measSN3}) and (\ref{cardSN}), for sufficiently large $\mu$, 
\begin{align} \label{cont}
\text{card}(S^{\mu} \cap N^{\mu} \cap \mathbb Z^2)  \leq C\mu^{3s}.
\end{align}

If every disjoint sets $S^{\mu} \cap N^{\mu}_m \cap \mathbb Z^2$ were non-empty for all $m \in [0,J]$, with $J=\lfloor \mu^{1/2} \rfloor$,
then $\text{card}(S^{\mu} \cap N^{\mu} \cap \mathbb Z^2)$ would grow at least as fast as $\mu^{1/2}$ as $\mu\rightarrow \infty$, contradicting (\ref{cont})
since $0<s<1/6$. Consequently, for sufficiently large $\mu$, there exists $m_0\in [0,J]$ such that the set $S^{\mu} \cap N^{\mu}_{m_0} \cap \mathbb Z^2$ is empty. 
Therefore, by (\ref{Smusup}), the annulus $N^{\mu}_{m_0}$ does not contain two lattice points at a distance less than or equal to $\mu^{s/2}$.
Setting $\lambda = \mu + (m_0 + \frac{1}{2})\kappa$, we have
\begin{align} \label{Nm0}
N_{m_0}^{\mu} &= \{x\in \mathbb R^2:  \mu+ m_0 \kappa < |x|^2 \leq \mu + (m_0+1) \kappa\} \notag\\
&= \{x\in \mathbb R^2:  \lambda - \frac{\kappa}{2} < |x|^2 \leq \lambda +\frac{\kappa}{2}\}.
\end{align}
Observe that as $\mu \rightarrow \infty$, the ratio $\frac{\lambda}{\mu}$ approaches 1, and since $\kappa = \mu^s$, it follows that $\lim_{\lambda \rightarrow \infty} \frac{\lambda^s}{\kappa}=1$, where $0<s<\frac{1}{6}$. Hence, for sufficiently large values of $\lambda$, it holds that $\kappa \geq \frac{1}{2} \lambda^s$. Also, the half-open annulus defined in (\ref{Nm0}) can be easily adjusted to a closed annulus.

\end{proof}

\vspace{0.1 in}

\subsection{Modification of the HNSE outside the absorbing ball}  \label{subsec-modi}
To study the HNSE (\ref{HNSEb}), we define the phase space $$\mathbb H=\{u\in (L^{2}(\mathbb{T}^{2}))^{2}:\int_{\mathbb{T}^{2}}u\,dx=0,\; \nabla \cdot u=0\}.$$
In this subsection, the norm $\|\cdot\|_{\mathbb H}$ is denoted simply as $\|\cdot\|$.

Note that for $u\in (L^{2}(\mathbb{T}^{2}))^{2}$, it can be expressed as $u=\sum_{j\in 
\mathbb{Z}^{2}}\hat{u}_{j}e^{ij\cdot x}$ with $\hat{u}_{j}$ being the Fourier
coefficients. Thus, $\int_{\mathbb{T}^{2}}udx=0$ is equivalent to $\hat{u}_{0}=0$. Therefore, if $u\in \mathbb H$, then $u=\sum_{j\in \mathbb{Z}^{2}\backslash \{0\}}\hat{u}_{j}e^{ij\cdot x}$.
The Helmholtz-Leray orthogonal projection operator is denoted as $P_{\sigma }:(L^{2}(\mathbb{T}^{2}))^{2}\rightarrow \mathbb H$ and the Stokes operator as $A=-P_{\sigma }\Delta$. In the periodic space, it is known that $Au=-P_{\sigma }\Delta u=-\Delta u$ for all $u\in \mathcal{D}(A)$.
The operator $A^{-1}$ is a self-adjoint, positive-definite, compact operator mapping from $\mathbb H$ to $\mathbb H$. The Sobolev space is defined as $H^{s}(\mathbb T^2)=\left\{ u\in \mathbb H:\|u\|_{H^{s}}^{2}=\sum_{j\in \mathbb{Z}^{2}\backslash \{0\}} |j|^{2s}|\hat{u}_{j}|^{2}<\infty \right\}$, for any $s\in \mathbb R$.

For $u, v\in H^1(\mathbb T^2)$, the bilinear form $B(u,v)=P_{\sigma }((u\cdot \nabla )v)$ is defined. 
Thus, the HNSE (\ref{HNSEb}) can be equivalently written in $\mathbb H$ as
\begin{equation} \label{BNSE}
\partial_t u+\nu A^{\beta}u+B(u,u)=f,    \;\; \text{for} \;\;\beta>\frac{17}{12}.
\end{equation}

\vspace{0.1 in}

\begin{proposition}
\label{gl_atr}
Assume $f\in H^{\frac{1}{6}}(\mathbb T^2)$. Let $S\left( t\right) : \mathbb H \rightarrow \mathbb H$ be the solution semigroup generated by the HNSE (\ref{BNSE}) with $\beta > \frac{17}{12}$. Then, $S(t)$ possesses an absorbing ball in $H^{3+\epsilon}(\mathbb T^2)$ of radius $\rho$, such that $\|S(t) u_0\|_{H^{3+\epsilon}} \leq \rho$ for $t\geq t_1(\|u_0\|)$, where $\epsilon= 2\beta-\frac{17}{6}>0$.
\end{proposition}

\vspace{0.1 in}

\begin{remark}
The existence of an absorbing ball for the 2D NSE is well-established in the literature, with detailed proofs available in texts like \cite{Robinson, Temam}. For equation (\ref{BNSE}), which includes hyper-viscosity, the proof adheres to the classical approach and is omitted here for brevity. Please note the relationship between the regularity of the forcing term $f\in H^{\frac{1}{6}}(\mathbb T^2)$ and the viscous term $\nu A^{\beta}u$, illustrated by the condition $\frac{1}{6}+2\beta >3$. 
This ensures that all trajectories of the dynamics enter an absorbing ball in $H^{3+\epsilon}$ at large time.
\end{remark}

\vspace{0.1 in}

Inertial manifolds are fundamentally concerned with the dissipative dynamics as $t\rightarrow \infty$. Consequently, we can adapt the PDE outside the absorbing ball by truncating the nonlinearity, as explained in Remark \ref{remk-pre}.

To truncate the nonlinearity, we introduce a smooth cut-off function, denoted $\theta$, belonging to $C_{0}^{\infty }(\mathbb{C})$. This function satisfies $\theta (\xi )=\xi$ when $|\xi |\leq 1$, and maintains $|\theta (\xi )|\leq 2$ for all $\xi \in \mathbb{C}$. Further, we define a vector-valued cut-off function $\vec{\theta}(\xi)=(\theta (\xi _{1}),\theta (\xi _{2}))$ for any $\xi =(\xi _{1},\xi _{2})$ in $\mathbb{C}^{2}$.

Recall that the Helmholtz-Leray orthogonal projector, denoted $P_{\sigma }$, from $(L^{2}(\mathbb{T}^{2}))^{2}$ to $\mathbb H$ is defined as
$P_{\sigma }u=\sum_{j\in \mathbb{Z}^{2}\backslash \{0\}}P_{j}\hat{u}_{j}e^{ij\cdot x}$, where $P_{j}$ are $2 \times 2$ matrices given by
$P_{j}=\frac{1}{|j|^{2}}\left( 
\begin{array}{ccc}
j_{2}^{2} & -j_{1}j_{2} \\ 
-j_{1}j_{2} & j_{1}^{2}
\end{array}
\right)$.

Let $\rho$ represent the radius of the absorbing ball in $H^{3+\epsilon}$ for the HNSE (\ref{BNSE}), as indicated in Proposition \ref{gl_atr}.
In line with Kostianko's approach \cite{Kostianko}, we define the operator $W: \mathbb H\rightarrow \mathbb H$ as
\begin{equation}
W(u)=\sum_{j\in \mathbb{Z}^{2}\backslash \{0\}}\frac{\rho }{|j|^{3+\epsilon}}P_{j} \vec{\theta}\left( \frac{|j|^{3+\epsilon}\hat{u}_{j}}{\rho }\right) e^{ij\cdot x}.
\label{defU}
\end{equation}
This operator is crucial for modifying the nonlinearity $B(u,u)$ in (\ref{BNSE}) outside the absorbing ball in $H^{3+\epsilon}$. Specifically, we replace $B(u,u)$ with $B(W(u),W(u))$, aiming to establish the existence of IMs in $\mathbb H$ for the ``prepared" equation of (\ref{BNSE}), which is
\begin{equation}
\partial_t u+ \nu A^{\beta}u+B(W(u),W(u))=f,  \;\; \text{for} \;\;\beta>\frac{17}{12}. \label{prepared}
\end{equation}
It is important to note that the following lemma ensures that both the original (\ref{BNSE}) and the \textquotedblleft prepared" (\ref{prepared}) equations exhibit identical large-time behaviors within the absorbing ball in $H^{3+\epsilon}(\mathbb{T}^{2})$.

\begin{lemma}
\label{lem1}

With $\epsilon= 2\beta-\frac{17}{6}>0$, the function $W: \mathbb H\rightarrow \mathbb H$ possesses these properties:
\begin{enumerate}
\item $W(u)=u$ provided $\|u\|_{H^{3+\epsilon}} \leq \rho$.

\item 

$W$ acts as a regularization operator, mapping $\mathbb H$ to $H^2$, and there exists a constant $C$ such that
\begin{equation}
\|W(u)\|_{H^{2}} \leq C,\text{\ \ for all\ \ }u\in \mathbb H.  \label{g17}
\end{equation}
Moreover, the map $W: \mathbb H\rightarrow H^2$ is continuous.

\item $W$ is Gateaux differentiable from $\mathbb H$ to $\mathbb H$. Its derivative $W^{\prime }$ is expressed as
\begin{equation}
W^{\prime }(u)v=\sum_{j\in \mathbb{Z}^{2}\backslash \{0\}}P_{j}\vec{\theta}
^{\prime }\left( \frac{|j|^{3+\epsilon}\hat{u}_{j}}{\rho }\right) \hat{v}
_{j}e^{ij\cdot x},\text{\ \ for\ \ }u,v\in \mathbb H.  \label{Gat}
\end{equation}
Furthermore, there exists a constant $L_{1}>0$ such that $\|W^{\prime }(u)\|_{\mathcal{L}(\mathbb H, \mathbb H)}\leq L_{1}$ for all $u\in \mathbb H$.
The map $u\mapsto W^{\prime }(u)v$ is continuous from $\mathbb H$ to $\mathbb H$ for each $v\in \mathbb H$.
\end{enumerate}
\end{lemma}

\begin{proof}

This proof closely mirrors that of Lemma 3.4 in \cite{GG}. Also see \cite{Kostianko}. To show that $W(u)=u$ when $\|u\|_{H^{3+\epsilon}} \leq \rho$, we consider 
$\|u\|_{H^{3+\epsilon}}^{2}=\sum_{j\in \mathbb{Z}^{2}\backslash \{0\}}|j|^{6+2\epsilon}|\hat{u}_{j}|^{2}\leq \rho ^{2}$. This implies $\frac{|j|^{3+\epsilon}|\hat{u}_{j}|}{\rho }\leq 1$ for all $j\in \mathbb{Z}^2\backslash \{0\}$. Since $\theta(\xi)=\xi$ for $|\xi|\leq 1$, it follows that $\vec{\theta}\left( \frac{|j|^{3+\epsilon}\hat{u}_{j}}{\rho }\right) =\frac{|j|^{3+\epsilon}\hat{u}_{j}}{\rho }$ for all $j\in \mathbb{Z}^{2}\backslash \{0\}$. Therefore, from (\ref{defU}), we conclude that $W(u)=u$ if $\|u\|_{H^{3+\epsilon}}\leq \rho$.

To establish that $\|W(u)\|_{H^{2}} \leq C$ for any $u\in \mathbb H$, we use the property $|\theta (\xi )|\leq 2$ for all $\xi \in \mathbb{C}$ along with (\ref{defU}). This yields the inequality
\begin{equation*}
\|W(u)\|_{H^2}^{2} \leq C\sum_{j\in \mathbb{Z}
^2 \backslash \{0\}}|j|^{4}\frac{\rho ^{2}}{|j|^{6+2\epsilon}}\leq C\rho
^{2}\sum_{j\in \mathbb{Z}^{2}\backslash \{0\}}\frac{1}{|j|^{2+2\epsilon}}\leq C(\epsilon,\rho).
\end{equation*}

For the proofs of the remaining parts, please refer to Lemma 3.4 in \cite{GG}.
\end{proof}

\vspace{0.1 in}

Without loss of generality, let us assume the viscosity $\nu=1$. The ``prepared" HNSE (\ref{prepared}) can be expressed in the form of the abstract equation (\ref{a1}) as:
\begin{equation}   \label{n1''}
\partial_t u+A^{\beta}u+A^{1/2}\mathscr F(u)=f,  \;\; \text{where}\;\;  \mathscr F(u)= A^{-1/2}B(W(u),W(u)).
\end{equation}

Consider vectors $u=(u_{1},u_{2})$, $v=(v_{1},v_{2})$, and $w=(w_{1},w_{2})$ in $\mathbb H$. We define the trilinear form as follows:
\begin{equation}
b(u,v,w)=(B(u,v),w)=
\sum_{m,n=1}^{2}\int_{\mathbb{T}^{2}} \big(u_{m}\frac{\partial v_{n}}{\partial x_m}\big)w_{n}dx,  \label{n10}
\end{equation}
applicable whenever the integrals in (\ref{n10}) are well-defined.

The following proposition has been proved in \cite{GG}. Please see Proposition 3.5 in \cite{GG}.

\begin{proposition}
\label{prop4} 
Let $\mathscr F$ be the operator as defined in (\ref{n1''}). Then $\mathscr F$ is uniformly bounded from $\mathbb H$ to $H^2$, i.e,
there exists a constant $C>0$ such that $\|\mathscr F(u)\|_{H^2}\leq C$ for all $u\in \mathbb H$.
Moreover, $\mathscr F$ is Gateaux differentiable from $\mathbb H$ to $\mathbb H$, with its
derivative $\mathscr F^{\prime }$ expressed as
\begin{equation}
(\mathscr F^{\prime }(u)v,w)=-b(W(u),A^{-1/2}w,W^{\prime }(u)v)-b(W^{\prime}(u)v,A^{-1/2}w,W(u)),  \label{g4}
\end{equation}
for $u$, $v$, $w\in \mathbb H$. In addition, there exists a constant $L>0$ such that 
$\|\mathscr F^{\prime }(u)\|_{\mathcal{L}(\mathbb H,\mathbb H)}\leq L$ for all $u\in \mathbb H$.
\end{proposition}

\vspace{0.1 in}

\subsection{Verification of the spatial averaging condition}

We now turn our attention to verifying the spatial averaging condition (\ref{averaging}) from Lemma \ref{thm1}, for the ``prepared" HNSE (\ref{prepared}). This verification will enable us to apply Corollary \ref{cor1} and consequently assert the existence of an IM for the ``prepared" HNSE. We consider the supercritical case $\frac{17}{12} < \beta < \frac{3}{2}$. In $\mathbb H$, the operator $A=-\Delta$ has eigenfunctions $\{e^{ij\cdot x}\}$ for all $j=(j_1,j_2) \in \mathbb Z^2 \backslash \{0\}$, 
with eigenvalues $\{\lambda_n\} = \{j_1^2 + j_2^2\}$. The eigenvalues $\{\lambda_n\}$ are sums of two squares, and we sort them as $0\leq \lambda_1 \leq \lambda_2 \leq \cdots$,  repeated according to their multiplicities.

\begin{proposition}
\label{prop1}  
For any $s \in (3-2\beta, \frac{1}{6})$ with $\frac{17}{12} < \beta < \frac{3}{2}$, and given the operator $\mathscr F$ as defined in (\ref{n1''}), there exist arbitrarily large $\lambda_N$ and $k \geq c\lambda_N^s$ such that
\begin{equation*}
\|\mathcal I_{k,N}\mathscr F^{\prime }(u)\mathcal I_{k,N}\|_{\mathcal L(\mathbb H, \mathbb H)} \leq   \frac{1}{16} \lambda_N^{ -   \frac{1}{2}(3-2\beta)} ,
\end{equation*}
for any $u \in \mathbb H$.
\end{proposition}

\begin{proof}
We denote the $L^p(\mathbb T^2)$-norm as $\|\cdot \|_p$, where $2\leq p\leq \infty$. Consider vectors $u=(u_{1},u_{2})$, $v=(v_{1},v_{2})$, and $w=(w_{1},w_{2})$ in $\mathbb H$. Using Proposition \ref{prop4}, we express:
\begin{align}
& (\mathcal I_{k,N}\mathscr F^{\prime }(u)\mathcal I_{k,N}v, \,w)
=(\mathscr F^{\prime}(u)\mathcal I_{k,N}v, \,\mathcal I_{k,N}w)  \label{n8}  \notag \\
& =-b(W(u), \, A^{-1/2}\mathcal I_{k,N}w, \,W^{\prime }(u)\mathcal I_{k,N}v)
-b(W^{\prime}(u)\mathcal I_{k,N}v, \, A^{-1/2}\mathcal I_{k,N}w, \, W(u)). 
\end{align}

We denote the $n$-th component of the vector $W(u)$ as $W(u)_{n}$ and the $n$-th component of the vector $W^{\prime }(u)v$ as $[W^{\prime }(u)v]_{n}$, where $n=1,2$. By (\ref{n10}), we obtain
\begin{align*}
& b(W(u), \, A^{-1/2} \mathcal I_{k,N}w, \, W^{\prime }(u)\mathcal I_{k,N}v) 
= \sum_{m,n=1}^{2}\int_{\mathbb{T}^2} W(u)_m \frac{\partial(A^{-\frac{1}{2}} \mathcal I_{k,N}w_n  )}{\partial x_m} [W^{\prime }(u)\mathcal I_{k,N}v]_n dx
\notag\\
&=\sum_{m,n=1}^{2}\int_{\mathbb{T}^2}\left[\mathcal I_{k,N}\left(W(u)_{m}\mathcal I_{k,N}[W^{\prime }(u)v]_{n}\right) \right]  \frac{\partial(A^{-\frac{1}{2}}w_n  )}{\partial x_m}dx.
\end{align*}
Applying the Cauchy-Schwarz inequality, we deduce that
\begin{equation*}
\left|b(W(u), \, A^{-1/2}\mathcal I_{k,N}w, \, W^{\prime }(u)\mathcal I_{k,N}v)\right| \leq
\|w\|_2 \sum_{m,n=1}^{2}\| \mathcal I_{k,N}\left( W(u)_{m}\mathcal I_{k,N}[U^{\prime
}(u)v]_{n}\right)\|_2.
\end{equation*}
The same estimate holds for the second term on the right-hand side of (\ref%
{n8}). Therefore, 
\begin{equation}
(\mathcal I_{k,N}\mathscr F^{\prime }(u)\mathcal I_{k,N}v, \,w)\leq
2 \|w\|_2 \sum_{m,n=1}^{2}\| \mathcal I_{k,N}\left( W(u)_{m}\mathcal I_{k,N}[W^{\prime
}(u)v]_{n}\right) \|_2, \label{n8'}
\end{equation}
for any $u,v,w\in \mathbb H$.

We aim to estimate the right-hand side of inequality (\ref{n8'}). 
Consider $\varphi $, $\psi \in L^{2}(\mathbb{T}^2)$, both having a zero mean value.
For $r>0$, we decompose $\varphi$ into high and low modes: $\varphi =\varphi _{>r}+\varphi _{<r}$ where
$\varphi _{>r} =\sum_{|j|>r}\hat{\varphi}_{j}e^{ij\cdot x}$
and $\varphi _{<r} =\sum_{1\leq |j|\leq r}\hat{\varphi}_{j}e^{ij\cdot x}$.
Then
\begin{equation}
\mathcal I_{k,N}(\varphi \mathcal I_{k,N}\psi )=\mathcal I_{k,N}(\varphi _{>r}\mathcal I_{k,N}\psi
)+\mathcal I_{k,N}(\varphi _{<r}\mathcal I_{k,N}\psi ).  \label{n4}
\end{equation}

By Lemma \ref{thm2D}, for any $s\in (0,\frac{1}{6})$, there exist arbitrarily large $\lambda >0$ and $k \geq C\lambda^s$ satisfying:
whenever $|n|^{2},|l|^{2}\in \lbrack \lambda - k ,\lambda + k]$ with
distinct $n$ and $l\in \mathbb{Z}^{2}$, it holds that $|n-l|>  \lambda^{s/2}$. Therefore, we can select arbitrarily large $\lambda
_{N}>0$ and $k \geq c\lambda_{N}^s$ satisfying: whenever 
$|n|^{2},|l|^{2}\in \lbrack \lambda _{N}-k,\lambda_{N}+k]$ with distinct $n$
and $l\in \mathbb{Z}^{2}$, we have $|n-l|> r= \lambda^{s/2}_N$. This leads to  
\begin{equation}
\mathcal I_{k,N}(\varphi _{<r}\mathcal I_{k,N}\psi )=\sum_{\lambda _{N}-k\leq |n|^{2}\leq
\lambda _{N}+k}\Big( \sum_{\substack{ \lambda _{N}-k\leq |l|^{2}\leq
\lambda _{N}+k  \\ 1\leq |n-l|\leq r}}\hat{\varphi}_{n-l}\hat{\psi}
_{l}\Big) e^{in\cdot x}=0.  \label{n5}
\end{equation}
From (\ref{n4}) and (\ref{n5}), for the chosen $N$ and $k$, it follows that
\begin{equation*}
\mathcal I_{k,N}(\varphi \mathcal I_{k,N}\psi )=\mathcal I_{k,N}(\varphi _{>r}\mathcal I_{k,N}\psi ).
\end{equation*}
Applying Agmon's inequality, and assuming $\varphi \in H^2(\mathbb T^2)$ with zero mean, we deduce
\begin{equation}
\|\mathcal I_{k,N}(\varphi \mathcal I_{k,N}\psi )\|_2
\leq \|\varphi _{>r}\mathcal I_{k,N}\psi\|_2
\leq \|\varphi_{>r}\|_{\infty} \|\psi\|_2
\leq C\|\varphi _{>r}\|_2^{\frac{1}{2}} \|\varphi\|_{H^2}^{\frac{1}{2}} \|\psi\|_2.  \label{n6}
\end{equation}
Note that
\begin{equation}
\|\varphi _{>r}\|^{2}_2 =\sum_{|j|>r}|\hat{\varphi}_{j}|^{2}=\sum_{|j|>r}\frac{1}{|j|^{4}}|j|^{4}|\hat{\varphi}_{j}|^{2}
\leq \frac{1}{r^{4}}\|\varphi\|_{H^2}^{2}.  \label{n7}
\end{equation}
Thus, combining (\ref{n6}) and (\ref{n7}), we arrive at
\begin{equation}
\|\mathcal I_{k,N}(\varphi \mathcal I_{k,N}\psi )\|_2
\leq \frac{C}{r}\|\varphi\|_{H^2} \|\psi\|_2.  \label{n9}
\end{equation}

Referring to (\ref{n8'}) and (\ref{n9}), we conclude that
\begin{equation}
(\mathcal I_{k,N}\mathscr F^{\prime }(u)\mathcal I_{k,N}v,w)\leq \frac{C\|w\|_2}{r} \sum_{m,n=1}^{2}\|W(u)_{m}\|_{H^2} \|[W^{\prime }(u)v]_{n}\|_2.  \label{n18}
\end{equation}
Given Lemma \ref{lem1}, which states $\|W(u)\|_{H^{2}}\leq C$ for all $u\in \mathbb H$ and $\|W^{\prime }(u)v\|_2 \leq L_{1}\|v\|_2$, we infer from (\ref{n18}) that 
\begin{equation*}
(\mathcal I_{k,N}\mathscr F^{\prime }(u)\mathcal I_{k,N}v,w)\leq \frac{CL_{1}}{r}\|v\|_2  \|w\|_2,
\text{\ \ for all\ \ }u, v, w \in \mathbb H.
\end{equation*}
Since $r= \lambda^{s/2}_N$, it follows that
\begin{equation*}
\|\mathcal I_{k,N}\mathscr F^{\prime }(u)\mathcal I_{k,N}v\|_2
\leq \frac{CL_{1}}{r}\|v\|_2\leq
CL_{1} \lambda_N^{-s/2}\|v\|_2
\leq  \frac{1}{16} \lambda_N^{ -   \frac{1}{2}(3-2\beta)} \|v\|_2, \text{\ \ for all\ \ }u,v\in \mathbb H,
\end{equation*}
for sufficiently large $\lambda_N$, given that $s > 3-2\beta$.
\end{proof}

\vspace{0.1 in}

\section{Discussion}

In this section, we explore the underlying motivation of our research. The paper \cite{GG} established the existence of IMs for the HNSE when the exponent $\beta \geq \frac{3}{2}$ in both 2D and 3D periodic domains. However, a lower value for $\beta$ is intuitively anticipated in 2D compared to 3D, for the IM problem. As an illustration, for the HNSE, global regularity and global attractors are available when $\beta \geq 1$ in 2D, as opposed to $\beta \geq \frac{5}{4}$ in 3D. Our work addresses this conjecture by successfully showing that the value of $\beta$ can be reduced below the critical threshold of $\frac{3}{2}$ for the IM problem in 2D. Looking ahead, our future objectives include exploring the potential for further reducing the value of $\beta$. In fact, gaining a deeper understanding of the distribution of lattice points in annular regions may prove beneficial. The investigation of sparse distributions of lattice points in annuli or on circles also presents interesting problems in number theory in their own right.

The concept of IM is designed to provide a robust analytical framework for understanding the finite-dimensionality characteristic of the asymptotic dynamics in dissipative PDEs. The dimension of an IM is defined in the traditional, topological sense as applied to manifolds. A significant gap in our understanding, however, is the uncertainty surrounding the existence of an IM for the NSE. In contrast, it is established that the 2D NSE possesses a finite-dimensional global attractor. The dimensions of this attractor are quantified using either fractal or Hausdorff dimensions. The dimension of the attractor serves as a valuable indicator of the degrees of freedom involved in the system's asymptotic dynamics. Nevertheless, it is important to recognize that the structure of a finite-dimensional attractor can be exceedingly complex. Therefore, both in theoretical and practical contexts, the dimension of the IM is often regarded as the ``true" dimension of the long-term dynamics of a dissipative PDE.

Also, it is worth mentioning that the discovery of a finite number of determining modes and nodes in the 2D NSE suggests a limited number of degrees of freedom for the asymptotic dynamics of turbulence. Foias and Frodi \cite{FP} originally introduced the concept of ``determining modes", which states that if two solutions of the NSE converge as $t\rightarrow \infty$ in the projection onto the first $N$ Fourier modes, they will also converge in their entirety. Analogously, ``determining nodes" refer to a finite set of points within the domain that can be used in place of Fourier modes for this convergence. However, it is important to emphasize that the existence of determining modes does not necessarily mean that these lower modes determine the asymptotic solution of the NSE, thus leaving the question of the finite-dimensionality of NSE asymptotic dynamics open for further investigation.

In the field of analyzing the NSE and Euler equations, there have been significant advancements concerning the non-uniqueness of weak solutions \cite{DS, BV, ABC}, the phenomenon of anomalous dissipation and Onsager's conjecture \cite{Isett, BD}, and the formation of singularities \cite{Elgindi, CH}. Equally crucial is the finite-dimensional nature of the asymptotic dynamics of the NSE, which is important for deepening our comprehension of fully developed turbulence. This work contributes to this critical area of study. Generally, regarding evolutionary PDEs with inherent regularization mechanisms, a ``simpler" dynamical behavior is often expected in the long term. For instance, the well-known soliton resolution conjecture suggests that for many nonlinear dispersive PDEs, solutions with generic initial conditions should ultimately decompose into a limited number of solitons, each moving at distinct velocities, accompanied by a radiative term that diminishes over time, thereby illustrating the finite-dimensional asymptotic dynamics characteristic of nonlinear dispersive PDEs.

\end{document}